\documentclass[reqno,12pt]{amsart}
\usepackage{amscd}
\usepackage{amssymb}
\usepackage{amsmath}
\usepackage{amsfonts}
\usepackage{enumerate}
\usepackage{graphicx}
\usepackage{epic}
\usepackage[all,cmtip]{xy}
\usepackage{overpic}
\usepackage{subfig}
\usepackage{caption}
\usepackage{tikz}

\newtheorem{thm}{Theorem}[section]
\newtheorem{prop}[thm]{Proposition}
\newtheorem{lem}[thm]{Lemma}
\newtheorem{cor}[thm]{Corollary}

\theoremstyle{definition}
\newtheorem{Def}[thm]{Definition}

\newtheorem*{rem*}{Remark}

\newcommand{\kk}{\mathbf{k}}
\newcommand{\xx}{\mathbf{x}}

\newcommand{\Aa}{\mathcal {A}}
\newcommand{\BB}{\mathcal {B}}

\newcommand{\PP}{\mathcal {P}}

\newcommand{\FF}{\mathcal {F}}

\newcommand{\Zz}{\mathbb {Z}}

\def\w{\widetilde}

\def\xr{\xrightarrow}

\numberwithin{equation}{section}
\usepackage[colorlinks,linktocpage]{hyperref}
\hypersetup{citecolor=blue,linkcolor=blue}
\begin{document}
	\title[On the anisotropy and Lefschetz property]{On the anisotropy and Lefschetz property for PL-spheres}
	\author[F.~Fan]{Feifei Fan}
	\thanks{The author is supported by the National Natural Science Foundation of China (Grant Nos. 11801580, 11871284)}
	\address{Feifei Fan, School of Mathematical Sciences, South China Normal University, Guangzhou, 510631, China.}
	\email{fanfeifei@mail.nankai.edu.cn}
	\maketitle
	\begin{abstract}
A simplicial sphere $\Delta$ is said to be  generically
anisotropic over a field $\mathbb{F}$ if, for a certain purely transcendental field extension $\kk$ of $\mathbb{F}$, a certain
Artinian reduction $A$ of the face ring $\kk[\Delta]$ has the following property: For every nonzero homogeneous element $\alpha\in A$ of degree at most $(\dim\Delta+1)/2$, its square $\alpha^2$ is also nonzero.
The importance of this property is that the hard Lefschetz property for simplicial spheres can be derived from it.
A recent result of Papadakis and Petrotou shows that every simplicial sphere is generically anisotropic over any field of characteristic $2$.
In this paper, we give an equivalent condition of being generically anisotropic, and use it to present a simplified proof of Papadakis-Petrotou theorem for PL-spheres. We also prove that the simplicial spheres of dimension $2$ are generically anisotropic over any field $\mathbb{F}$.
	\end{abstract}
	\section{Introduction}\label{sec:introduction}
	\renewcommand{\thethm}{\arabic{thm}}
	This paper is devoted to the study of face rings of simplicial spheres, especially PL-spheres. The face ring $\kk[\Delta]$ of a simplicial complex $\Delta$ is a quotient of a polynomial ring over a field $\kk$. Its properties were investigated  by Hochster \cite{H75}, Reisner \cite{Rei76} and Stanley \cite{Sta75} in 1970's. Face rings play a key role in combinatorial commutative algebra, because their algebraic properties reflect many surprising combinatorial and topological properties of the associated simplicial complexes.  
	
	For example, the hard Lefschetz theorem for face rings of simplicial polytope, which was originally proved by Stanley \cite{S80}, and later by McMullen \cite{Mc93,Mc96}, verifies the necessity part of McMullen's conjecture \cite{Mc71} that posited a complete characterization of face numbers of simplicial polytopes. The sufficiency part was proved by Billera and Lee \cite{BL81}. Since then, one of the most important problems in the realm of face enumeration is whether or not McMullen's
	conditions extend to simplicial spheres or even homology spheres. This question has become known as the $g$-conjecture. The algebraic version of this conjecture, which implies the $g$-conjecture, posits that all homology spheres satisfy the hard Lefschetz theorem.
	
	Recently, a proof of the algebraic $g$-conjecture, due to Adiprasito, has appeared \cite{A18}. Adiprasito's proof is quite technical, because it involves many complicated geometric constructions. More recently, Papadakis and Petrotou \cite{PP20} gave another purely algebraic proof of this conjecture for the case that $\kk$ has characteristic $2$. 
	Then, Adiprasito, Papadakis, and Petrotou posted a new preprint \cite{APP21}, which provides far reaching generalizations of these algebraic results to much
	more general simplicial complexes such as normal pseudomanifolds.
A more surprising result is that Papadakis and Petrotou \cite{PP20} established an essentially stronger algebraic properties than the Lefschetz property for simplicial spheres, which says that every simplicial sphere is generically anisotropic over a filed of characteristic $2$ (see Subsection \ref{subsec:anisotropy}).
In this paper, we mainly follow the idea in \cite{PP20}, and focus on the face rings of PL-spheres.
	
This paper is organized as follows: in Section \ref{sec:preliminaries}, we introduce the basic notions and collect the previous results about face rings. In Section \ref{sec:equivalence}, we establish an equivalent but simpler condition of being generically anisotropic (Lemma \ref{lem:aniso equiv}), and use it to generalize a result in \cite{PP20} that $1$-dimensional simplicial spheres are generically anisotropic over any field $\mathbb{F}$, to stacked spheres (Theorem \ref{thm:0-move}). Section \ref{sec:char 2} is devoted to giving a simple proof of Papadakis-Petrotou theorem (Theorem \ref{thm:PP}) for PL-spheres by using a bistellar move argument (Theorem \ref{thm:bistellar}). In Section \ref{sec:2-sphere}, we prove that the generic anisotropy of odd dimensional PL-spheres implies the generic anisotropy of even dimensional PL-spheres of one higher dimension (Theorem \ref{thm:general char}). As a consequence, we see that every $2$-dimensional simplicial sphere is generically anisotropic over any field $\mathbb{F}$ (Corollary \ref{cor:2-sphere}), solving a conjecture in \cite{PP20,APP21} in this special case.

	\section{Preliminaries}\label{sec:preliminaries}
	\renewcommand{\thethm}{\thesection.\arabic{thm}}
	\subsection{Basic concepts}\label{subsec:notation}
	Throughout this paper, we assume that $\Delta$  is a  simplicial complex.
	If $\Delta$ has $m$ vertices, we usually identify the vertices of $\Delta$ with $[m]=\{1,\dots,m\}$.
	We refer to $i$-dimensional faces as \emph{$i$-faces}, and denote by $\FF_i(\Delta)$ the set of $i$-faces  of $\Delta$.
	A simplicial complex is \emph{pure} if all of its facets (maximal faces) have the same dimension.
	By $\Delta^{m-1}$ we denote the simplicial complex consisting of all subsets of $[m]$; its boundary  $\partial\Delta^{m-1}$ will be the subcomplex of all proper subsets of $[m]$. By abuse of notation, sometimes the symbol $\sigma$ will be used ambiguously to denote a face $\sigma\in\Delta$ and also the simplicial complex consisting of $\sigma$ and all its faces.

For a $(d-1)$-dimensional simplicial complex $\Delta$, the \emph{$f$-vector} of $\Delta$ is
\[(f_0,f_1,\dots,f_{d-1}),\]
where $f_i$ is the number of the $i$-dimensional faces of $\Delta$.
Sometimes it is convenient to set $f_{-1}=1$ corresponding to the empty set. The \emph{$h$-vector} of $\Delta$ is the integer vector
$(h_0,h_1,\dots,h_d)$ defined from the equation
\[
h_0t^d+\cdots+h_{d-1}t+h_d=f_{-1}(t-1)^d+f_0(t-1)^{d-1}+\cdots+f_{d-1}.
\]
	
	The \emph{link} and \emph{star} of a face $\sigma\in\Delta$ are respectively the subcomplexes
	\[\begin{split}
		\mathrm{lk}_\sigma\Delta=&\{\tau\in\Delta:\tau\cup\sigma\in\Delta,\tau\cap\sigma=\emptyset\};\\
		\mathrm{st}_\sigma\Delta=&\{\tau\in\Delta:\tau\cup\sigma\in\Delta\}.
	\end{split}\]
	
	The \emph{join} of two simplicial complexes $\Delta$ and $\Delta'$, where the vertex set $\FF_0(\Delta)$ is disjoint from $\FF_0(\Delta')$, is the simplicial complex
	\[\Delta*\Delta'=\{\sigma\cup\sigma':\sigma\in\Delta, \sigma'\in\Delta'\}.\]
	
	Let $\Delta$ be a pure simplicial complex of dimension $d$ and $\sigma\in\Delta$ a $(d-i)$-face such that $\mathrm{lk}_\sigma\Delta=\partial\Delta^i$ and the subset $\tau=\FF_0(\Delta^i)\subset \FF_0(\Delta)$ is not a face of $\Delta$. Then the operation $\chi_\sigma$ on $\Delta$ defined by
	\[\chi_\sigma\Delta=(\Delta\setminus\mathrm{st}_\sigma\Delta)\cup(\partial\sigma*\tau)\]
	is called a \emph{bistellar $i$-move}. Obviously we have $\chi_\tau\chi_\sigma\Delta=\Delta$. Two pure simplicial complexes  are \emph{bistellarly equivalent}  if one is transformed to another by a finite sequence of bistellar moves.
	
	A \emph{piecewise linear (PL for short) $d$-sphere} is a simplicial complex which has a common subdivision with $\partial\Delta^{d+1}$. A \emph{PL $d$-manifold} is a simplicial complex $\Delta$ of dimension $d$ such that $\mathrm{lk}_\sigma\Delta$ is a PL-sphere of dimension  $d-|\sigma|$ for every nonempty face $\sigma\in\Delta$, where $|\sigma|$ denotes the cardinality of $\sigma$.
	It is obvious that two bistellarly equivalent PL-manifolds are PL homeomorphic, i.e., they have a common subdivision. The following fundamental result shows that the converse is also true.
	\begin{thm}[Pachner {\cite[(5.5)]{P91}}]\label{thm:pachner}
		Two PL-manifolds are bistellarly equivalent if and only if they are PL homeomorphic.
	\end{thm}

	A simplicial complex $\Delta$ is called a \emph{$\kk$-homology $d$-sphere} ($\kk$ is a field) if
 \[H_*(\mathrm{lk}_\sigma\Delta;\kk)=H_*(S^{d-|\sigma|};\kk)\quad \text{for all }\sigma\in\Delta \text{ (including $\sigma=\emptyset$)}.\]
(Remark: Usually, the terminology ``homology sphere"  means a manifold having the homology of a sphere. Here we take it in a more relaxed sense than its usual meaning.)
	
	\subsection{Face rings and l.s.o.p}\label{subsec:l.s.o.p.}
	For a field $\kk$, let $S=\kk[x_1,\dots,x_m]$ be the polynomial algebra with one generator for each
	vertex in $\Delta$. It is a graded algebra by setting $\deg x_i=1$.
	The \emph{Stanley-Reisner ideal} of $\Delta$ is
	\[I_\Delta:=(x_{i_1}x_{i_2}\cdots x_{i_k}:\{i_1,i_2,\dots,i_k\}\not\in\Delta)
	\]
	The \emph{Stanley-Reisner ring} (or \emph{face ring}) of $\Delta$ is the quotient \[\kk[\Delta]:=S/I_\Delta.\]
	Since $I_\Delta$ is a monomial ideal, the quotient ring $\kk[\Delta]$ is graded by degree.
	
	For a face $\sigma=\{i_1,\dots,i_k\}\in\FF_{k-1}(\Delta)$, denote by $\xx_\sigma=x_{i_1}\cdots x_{i_k}\in\kk[\Delta]$ the face monomial corresponding to $\sigma$.
	
A sequence $\Theta=(\theta_1,\dots,\theta_d)$ of $d=\dim\Delta+1$ linear forms in $S$ is called a \emph{linear
	system of parameters} (or \emph{l.s.o.p.} for short) if
	\[\kk(\Delta;\Theta):=\kk[\Delta]/(\Theta)\]
	has Krull dimension zero, i.e., it is a finite-dimensional $\kk$-space.
	The quotient ring $\kk(\Delta;\Theta)$ is an \emph{Artinian reduction} of $\kk[\Delta]$. We will use the simplified notation $\kk(\Delta)$ for $\kk(\Delta;\Theta)$ whenever it creates no confusion, and write the component
	of degree $i$ of $\kk(\Delta)$ as $\kk(\Delta)_i$. For a subcomplex $\Delta'\subset\Delta$, let $I$ be the ideal of $\kk[\Delta]$ generated by faces in $\Delta\setminus\Delta'$, and denote $I/I\Theta$ by $\kk(\Delta,\Delta';\Theta)$ or simply $\kk(\Delta,\Delta')$.
	
	A linear sequence $\Theta=(\theta_1,\dots,\theta_d)$ is an l.s.o.p if and only if the restriction $\Theta_\sigma=r_\sigma(\Theta)$ to each face $\sigma\in\Delta$ generates the polynomial algebra $\kk[x_i:i\in\sigma]$; here $r_\sigma:\kk[\Delta]\to\kk[x_i:i\in\sigma]$ is the projection homomorphism. If $\Theta$ is an l.s.o.p. for $\kk[\Delta]$, then $\kk(\Delta)$ is spanned by the face monomials (see \cite[II. Theorem 5.1.16]{BH98}). It is a general fact that if $\kk$ is an infinite field, then $\kk[\Delta]$ admits an l.s.o.p. (Noether normalization lemma). 
	
	Suppose $\Theta=(\theta_i=\sum_{j=1}^m a_{ij}x_j)_{i=1}^d$ is an l.s.o.p. for $\kk[\Delta]$. Then there is an associated $d\times m$ matrix $M_\Theta=(a_{ij})$.
Let $\boldsymbol{\lambda}_i=(a_{1i},a_{2i},\dots,a_{di})^T$ denote the column vector corresponding to the vertex $i\in[m]$. For any ordered subset $I=(i_1,\dots,i_k)\subset [m]$, the submatrix $M_\Theta(I)$ of $M_\Theta$ is defined to be
\[M_\Theta(I)=(\boldsymbol{\lambda}_{i_1},\dots,\boldsymbol{\lambda}_{i_k}).\]
	
	\subsection{Cohen-Macaulay and Gorenstein complexes}
	Throughout this subsection, $\kk$ is an infinite field of arbitrary characteristic.
	
	Let $\Delta$ be a simplicial complex of dimension $d-1$. The face ring $\kk[\Delta]$ is a \emph{Cohen-Macaulay ring} if for any l.s.o.p $\Theta=(\theta_1,\dots,\theta_d)$, $\kk[\Delta]$ is a free $\kk[\theta_1,\cdots,\theta_d]$ module. In this case, $\Delta$ is called a \emph{Cohen-Macaulay complex over $\kk$}.
	
	If $\Delta$ is Cohen-Macaulay, the following result of Stanley shows that the $h$-vector of $\Delta$ has a pure algebraic description.
	\begin{thm}[Stanley \cite{Sta75}]\label{thm:stanley}
		Let $\Delta$ be a $(d-1)$-dimensional Cohen-Macaulay complex and let $\Theta=
		(\theta_1,\dots,\theta_d)$ be any l.s.o.p. for $\kk[\Delta]$. Then
		\[\dim_\kk\kk(\Delta)_i=h_i(\Delta),\quad \text{for all } 0\leq i\leq d.\]
	\end{thm}
	
	The face ring $\kk[\Delta]$ is a \emph{Gorenstein ring} if  $\kk(\Delta;\Theta)$ is a Poincar\'e duality $\kk$-algebra for any l.s.o.p $\Theta$.
	In this case, $\Delta$ is called a \emph{Gorenstein complex over $\kk$}.
	Further, $\Delta$ is called \emph{Gorenstein*} if $\kk[\Delta]$ is Gorenstein and $\Delta$ is not a cone, i.e., $\Delta\neq\Delta^0*\Delta'$ for any $\Delta'$.
	
	These algebraic properties of face rings have combinatorial-topological characterisations as follows.
	\begin{thm}\label{thm:algebraic property}
		Let $\Delta$ be a simplicial complex. Then
		\begin{enumerate}[(a)]
			\item {\rm(Reisner \cite{Rei76})} $\Delta$ is Cohen-Macaulay (over $\kk$) if and only if for all faces $\sigma\in\Delta$ (including $\sigma=\emptyset$)
			and $i<\dim\mathrm{lk}_\sigma\Delta$, we have $\w H_i(\mathrm{lk}_\sigma\Delta;\kk)=0$.\vspace{8pt}
			
			\item {\rm(Stanley \cite[II. Theorem 5.1]{S96})} $\Delta$ is Gorenstein* (over $\kk$) if and only if it is a $\kk$-homology sphere.\vspace{8pt}
		\end{enumerate}
	\end{thm}
	
	\subsection{Lefschetz property of face rings}\label{subsec:lefschetz}
	The Lefschetz property of face rings is strongly connected to many topics in algebraic geometry, commutative algebra and combinatorics. For instance, the Lefschetz property for homology spheres is just the algebraic $g$-conjecture, and then implies the $g$-conjecture. Its general definition is as follows.
	
	\begin{Def}\label{def:lefschetz}
		Let $\Delta$ be a Cohen-Macaulay complex (over $\kk$). We say that $\kk[\Delta]$ (or simply $\Delta$)  has the \emph{Lefschetz property}
		if there is an Artinian reduction $\kk(\Delta;\Theta)$ of $\kk[\Delta]$ and a linear form
		$\omega\in\kk[\Delta]_1$ such that  the multiplication map
	\[\cdot\omega^{i}:\kk(\Delta;\Theta)_{j}\to\kk(\Delta;\Theta)_{j+i}\]
 is either surjective or injective for all $i\geq 1$ and $j\geq 0$.
	\end{Def}
In the case of a $\kk$-homology $(d-1)$-sphere $\Delta$, the Lefschetz property for $\Delta$ is equivalent to the property that there exists $(\Theta,\omega)$ such that the multiplication map	\[\cdot\omega^{d-2i}:\kk(\Delta;\Theta)_{i}\to\kk(\Delta;\Theta)_{d-i}\]
is an isomorphism for all $i\leq d/2$, known as the \emph{hard Lefschetz property}.

Note that the set of $(\Theta,\omega)$ in definition \ref{def:lefschetz} is
Zariski open, but it may be empty. The recent striking result by Adiprasito \cite{A18}, and by Adiprasito, Papadakis, and Petrotou \cite{APP21} shows that this set is nonempty for homology spheres.

	\begin{thm}[Hard Lefschetz \cite{A18,APP21}]\label{thm:hard lefschetz}
    Let $\kk$ be an infinite field. Then every $\kk$-homology sphere has the Lefschetz property.
	\end{thm}
	
\subsection{Anisotropy of face rings}\label{subsec:anisotropy}
As we mentioned in Section \ref{sec:introduction}, in order to prove the Lefschetz property for simplicial spheres, Papadakis and Petrotou \cite{PP20} established a stronger property for face rings of simplicial spheres.
Here is the formal definition.
\begin{Def}\label{def:anisotropy}
Let  $\Delta$ be a $\kk$-homology sphere of dimension $d-1$. An Artinian reduction $\kk(\Delta)$ of $\kk[\Delta]$ is said to be \emph{anisotropic} if for every nonzero element $u\in\kk(\Delta)_i$ with $i\leq d/2$, the square $u^2$ is also nonzero in $\kk(\Delta)_{2i}$. We call $\Delta$  anisotropic over $\kk$ if such an Artinian reduction exists.
\end{Def}

\begin{thm}[{\cite[Papadakis-Petrotou]{PP20}}]\label{thm:PP}
	If $\mathbb{F}$ is a field of characteristic $2$ and $\Delta$ is a $\mathbb{F}$-homology sphere, then there exists a purely transcendental field extension $\kk$ of $\mathbb{F}$ such that $\Delta$ is anisotropic over $\kk$.
\end{thm}
It turns out that being anisotropic is stronger than the Lefschetz property in the sence that if every simplicial sphere is anisotropic over $\kk$, then every simplicial sphere has the Lefschetz property over any infinite field $\kk'$ of the same characteristic as $\kk$ (see \cite[Section 9]{PP20}).

The transcendental field extension $\kk$ and the Artinian reduction $\kk(\Delta)$ in Theorem \ref{thm:PP} can be chosen as follows. Suppose that $\dim \Delta=d-1$, $\FF_0(\Delta)=[m]$. Set \[\kk=\mathbb{F}(a_{ij}:1\leq i\leq d,\,1\leq j\leq m),\]
 the field of rational functions on variables $a_{ij}$, 
and set $\theta_i=\sum_{i=1}^m a_{ij}$ for $1\leq i\leq d$. Then $\Theta=(\theta_i:1\leq i\leq d)$ is clearly an l.s.o.p. for $\kk[\Delta]$. For such $\kk$ and $\Theta$, \cite{PP20} shows that $\kk(\Delta;\Theta)$ is anisotropic. In this situation, $\Delta$ is also refered to be \emph{generically anisotropic over $\mathbb{F}$}.

\subsection{Canonical modules for homology balls} In this subsection we recall some results about the \emph{canonical module} (see \cite[I.12]{S96} for the definition) of $\kk[\Delta]$ when $\Delta$ is a homology ball.

Let $\Delta$ be a $\kk$-homology $(d-1)$-ball with boundary $\partial \Delta$. Then there is an exact sequence
\begin{equation}\label{eq:exact}
	0\to I\to \kk[\Delta]\to\kk[\partial\Delta]\to 0,
\end{equation}
where $I$ is the ideal of $\kk[\Delta]$ generated by all faces in $\Delta\setminus\partial\Delta$. By a theorem of Hochster \cite[Theorem  II.7.3]{S96} $I$ is the canonical module of $\kk[\Delta]$. Then from \cite[Theorem I.3.3.4 (d) and Theorem I.3.3.5 (a)]{BH98} we have the following proposition.
\begin{prop}\label{prop:pairing}
	Let $\Delta$ be a $\kk$-homology $(d-1)$-ball. Then there is a perfect bilinear pairing between the Artinian reductions
	\[
	\kk(\Delta)_i\times \kk(\Delta,\partial\Delta)_{d-i}\to \kk(\Delta,\partial\Delta)_d=\kk.
	\]
\end{prop}

\subsection{Lee's formula}\label{subsec:Lee's} 
In this subsection, we introduce a formula due to Lee that expresses non-square-free monomials in $\kk(\Delta)$ in terms of face monomials. First, we recall a useful result in \cite{PP20}.

\begin{lem}[{\cite[Corollary 4.5]{PP20}}]\label{lem:generator}
	Let $\Delta$ be a $(d-1)$-dimensional $\kk$-homology sphere or ball, $\Theta$ be an l.s.o.p. for $\kk[\Delta]$. Suppose that $\sigma_1$ and $\sigma_2$ are two ordered facets of $\Delta$, which have the same orientation in $\Delta$. Then \[\det(M_\Theta(\sigma_1))\cdot\xx_{\sigma_1}=\det(M_\Theta(\sigma_2))\cdot\xx_{\sigma_2}\] in $\kk(\Delta)_d$ or in $\kk(\Delta,\partial\Delta)_d$, respectively. 
\end{lem}

When $\Delta$ is a $(d-1)$-dimensional $\kk$-homology sphere or ball, $\kk(\Delta)_d$ or $\kk(\Delta,\partial\Delta)_d$ respectively, is an $1$-dimensional $\kk$-vector space, which is spanned by a facet. So each facet $\sigma\in\Delta$ defines a map \[\Psi_\sigma:\kk(\Delta)_d\text{ or } \kk(\Delta,\partial\Delta)_d\to \kk\]
such that for all $\alpha$ in $\kk(\Delta)_d$ or $\kk(\Delta,\partial\Delta)_d$,
\[\alpha=\Psi_\sigma(\alpha)\det(M_\Theta(\sigma))\xx_{\sigma}.\]
Lemma \ref{lem:generator} says that $\Psi_\sigma=\pm\Psi_\tau$ for  any two facets $\sigma,\tau\in\Delta$.
If we fix an orientation on $\Delta$, this map is independent of the choice of the oriented facet, giving a \emph{canonical function} $\Psi_\Delta:\kk(\Delta)_d\text{ or } \kk(\Delta,\partial\Delta)_d\to \kk$ (see \cite[Remark 4.6]{PP20}). In particular, if $\sigma$ is a facet of $\Delta$, then $\Psi_\Delta(\xx_\sigma)=1/\det(M_\Theta(\sigma))$.

To state Lee's formula, we will need the following notation. Under the assumption of Lemma \ref{lem:generator}, let $\mathbf{a}=(a_1,\dots,a_d)^T\in\kk^d$ be a vector such that every $d\times d$ minor of the  matrix $(M_\Theta\mid\mathbf{a})$ is nonsingular. For any ordered subset $I\subset[m]$ with $|I|=d$, let $A_I=\det(M_\Theta(I))$, and for any $i\in I$, denote by $A_I(i)$ the determinant of the matrix obtained from $M_\Theta(I)$ by replacing the column vector $\boldsymbol{\lambda}_i$ with $\mathbf{a}$.

\begin{thm}[Lee {\cite[Theorem 11]{Lee96}}]\label{thm:Lee}
	Let $\Delta$ be a $(d-1)$-dimensional $\kk$-homology sphere (resp. $\kk$-homology  ball), and fix an orientation on $\Delta$. Then for a monomial $x_{i_1}^{r_1}\cdots x_{k}^{r_k}\in\kk(\Delta)_d$ (resp.  $x_{i_1}^{r_1}\cdots x_{k}^{r_k}\in\kk(\Delta,\partial \Delta)_d$), $r_i>0$, we have
	\[\Psi_\Delta(x_{i_1}^{r_1}\cdots x_{k}^{r_k})=\sum_{F\in\FF_{d-1}(\mathrm{st}_\sigma\Delta)}\frac{\prod_{i\in\sigma}A_F(i)^{r_i-1}}{A_F\prod_{i\in F\setminus\sigma}A_F(i)},\]
	where $\sigma=\{i_1,\dots,i_k\}$ and the sum is over all ordered facets of $\mathrm{st}_\sigma\Delta$, which are compatible with the given orientation.
\end{thm}

\section{Equivalence of anisotropy}\label{sec:equivalence}
As we have seen in subsection \ref{subsec:anisotropy}, if $\mathbb{F}$ is a field of characteristic $2$ and $\Delta$ is a $\mathbb{F}$-homology $(d-1)$-sphere with $m$ vertices, then $\kk(\Delta;\Theta)$ is anisotropic for the field extension
\[\kk:=\mathbb{F}(a_{ij}:1\leq i\leq d,\,1\leq j\leq m)\]
and the l.s.o.p. $\Theta=(\theta_i=\sum_{i=1}^m a_{ij})_{i=1}^d$.
In fact, the field extension in Theorem \ref{thm:PP} can be chosen to be smaller than $\kk$, as we will see below.

Let 
\[\kk'=\mathbb{F}(a_{ij}:1\leq i\leq d,\,d+1\leq j\leq m),\]
and denote by $A$ the $d\times (m-d)$ matrix $(a_{ij})$. One easily sees that there is an l.s.o.p. $\Theta'$ for $\kk'[\Delta]$ such that $M_{\Theta'}=( I_d\mid A)$, where $I_d$ is the $d\times d$ identity matrix.

\begin{lem}\label{lem:aniso equiv}
	Suppose that $\Delta$ is a $\mathbb{F}$-homology ($\mathrm{char}\,\mathbb{F}$ is arbitrary) $(d-1)$-sphere with $m$ vertices. Let $\kk$, $\Theta$ and $\kk'$, $\Theta'$ be as above. 
	Then $\kk(\Delta;\Theta)$ is anisotropic if and only if $\kk'(\Delta;\Theta')$ is anisotropic.
\end{lem}
\begin{proof}
	``$\Rightarrow$".	There exists a matrix $N\in GL(d,\kk)$ such that $N M_\Theta=(I_d\mid B)$, where $B=(b_{ij})$ ($1\leq i\leq d,\,d+1\leq j\leq m$) is a $d\times (m-d)$ matrix  with entries $b_{ij}\in\kk$. Denote by $\Theta_0$ the l.s.o.p. corresponding to $(I_d\mid B)$. Clearly, the two ideals generated by $\Theta$ and $\Theta_0$ are the same. Let 
	\[\kk_0=\mathbb{F}(b_{ij}:1\leq i\leq d,\,d+1\leq j\leq m).\]
	Then $\kk_0$ is a subfield of $\kk$. One easily sees that $b_{ij}$ are algebraically independent elements over $\mathbb{F}$, so there is an isomorphism $\kk'\cong\kk_0$ given by $a_{ij}\mapsto b_{ij}$, and then an induced isomorphism $\kk'(\Delta;\Theta')\cong \kk_0(\Delta;\Theta_0)$. Since $\kk_0(\Delta;\Theta_0)\subset \kk(\Delta;\Theta_0)=\kk(\Delta;\Theta)$ 
	and $\kk(\Delta;\Theta)$ is anisotropic, any nonzero element 
	$u\in\kk_0(\Delta;\Theta_0)_i$ with $i\leq d/2$ satisfies $u^2\neq0$. Hence $\kk'(\Delta;\Theta')$ is anisotropic.
	
	``$\Rightarrow$". Pick an arbitrary order on the variables $a_{ij}$ for $1\leq i\leq d$, $1\leq j\leq d$, and rewrite them as $a_1,a_2,\dots,a_{d^2}$. Let $\kk_0=\kk'$, and recursively define $\kk_i=\kk_{i-1}(a_i)$, i.e. the field of fractions of $\kk_{i-1}[a_i]$,  for $1\leq i\leq d^2$. Hence there is a sequence of field extension \[\kk'=\kk_0\subset\kk_1\subset\cdots\subset\kk_{d^2}=\kk.\]
	 Let $\Theta_0=\Theta'$ be the l.s.o.p. for $\kk_0[\Delta]$. For $1\leq i\leq d^2$, if $a_{i}=a_{jk}$, then define an l.s.o.p. $\Theta_i$ for $\kk_i[\Delta]$ such that $M_{\Theta_i}$ is obtained from $M_{\Theta_{i-1}}$ by replacing the $(j,k)$-entry by $a_{jk}$. 
	
	We will prove that $\kk_{i}(\Delta;\Theta_i)$ are all anisotropic for $0\leq i\leq d^2$ by induction on $i$. The base case $i=0$ is just the assumption. For the induction step, set $R_i=\mathbb{F}[a_1,\dots,a_i]$, and denote by $\mathfrak{p}_i$ the prime ideal 
	\[\mathfrak{p}_i=\begin{cases}
		(a_i-1)\ &\text{if  $a_i=a_{kk}$ for some $1\leq k\leq d$},\\
		(a_i)\ &\text{otherwise.}
	\end{cases}\] 
	Then there is a ring homomorphism $\eta_i:(R_i)_{\mathfrak{p}_i}\to \kk_{i-1}$, where $(R_i)_{\mathfrak{p}_i}\subset\kk_i$ denote the localization of $R_i$ at $R-\mathfrak{p}_i$, given by $\eta_i(a_j)=a_j$ for $1\leq j\leq i-1$, and 
	\[\eta_i(a_i)=\begin{cases}
		1\ &\text{ if  $a_i=a_{kk}$ for some $k$},\\
		0\ &\text{ otherwise.}
	\end{cases}\]
	Given a nonzero element $\alpha\in\kk_{i}(\Delta;\Theta_i)_j$ with $j\leq d/2$, there exists a nonzero element $t\in\mathfrak{p}_i$ such that
	\[t\alpha\in(R_i)_{\mathfrak{p}_i}(\Delta;\Theta_i)\ \text{ and }\ 0\neq\eta_i(t\alpha)\in\kk_{i-1}(\Delta;\Theta_{i-1}).\]
	Since $\kk_{i-1}(\Delta;\Theta_{i-1})$ is anisotropic by induction, $(\eta_i(t\alpha))^2\neq0$. It follows that $t^2\alpha^2$ is not zero in $(R_i)_{\mathfrak{p}_i}(\Delta;\Theta_i)$, and then $0\neq\alpha^2\in\kk_{i}(\Delta;\Theta_i)$. So $\kk_{i}(\Delta;\Theta_i)$ is anisotropic.
\end{proof}

As an application of Lemma \ref{lem:aniso equiv}, we can obtain a result in \cite{PP20}. That is, every simplicial $1$-sphere is generically anisotropic over any field. Note that every simplicial $1$-sphere is obtained from $\partial \Delta^2$ by a sequence of bistellar $0$-moves, so this result is a corollary of the following more general theorem.

\begin{thm}\label{thm:0-move}
	Let $\Delta$ be a $\mathbb{F}$-homology $(d-1)$-sphere and suppose that $\Delta'$ is obtained from $\Delta$ via a bistellar $0$-move. Then
	$\Delta$ is generically anisotropic over $\mathbb{F}$ if and only if $\Delta'$ is  generically anisotropic over $\mathbb{F}$.
\end{thm}
\begin{proof}
	Suppose that $\FF_0(\Delta)=[m-1]$, $\FF_0(\Delta')=[m]$, and without loss of generality assume that  $\Delta'=\chi_\sigma\Delta$ for a facet $\sigma=[d]$.  Set 
	\begin{gather*}
		\kk=\mathbb{F}(a_{ij}:1\leq i\leq d,\,d+1\leq j\leq m-1),\\
		\kk'=\mathbb{F}(a_{ij}:1\leq i\leq d,\,d+1\leq j\leq m).
	\end{gather*}
Denote by $\Theta$  the l.s.o.p. for $\kk'[\Delta']$ such that $M_{\Theta}=(I_d\mid A)$, where $A=(a_{ij})$ is the $d\times(m-d)$ matrix. The restriction of $\Theta$ to $\kk[\Delta]$ is also denoted $\Theta$, and we will omit it from the notation of artinian reductions of face rings.

According to Lemma \ref{lem:aniso equiv}, it suffices to show that $\kk(\Delta)$ is anisotropic if and only if $\kk'(\Delta')$ is anisotropic. 
One direction is easy, since 
\[\kk(\Delta)_{\geq 1}=\kk(\Delta,\sigma)_{\geq 1}=\kk(\Delta',\mathrm{st}_{\{m\}}\Delta')_{\geq 1}\subset \kk'(\Delta')_{\geq 1}.\] 

The other direction takes more thought. Suppose that $\kk(\Delta)$ is anisotropic. Let $\kk_0=\kk$, and recursively define $\kk_i=\kk_{i-1}(a_{m,i})$ for $1\leq i\leq d$. Then $\kk=\kk_0\subset\kk_1\subset\cdots\subset\kk_{d}=\kk'$ is a sequence of field extension. Arguing as in the proof of Lemma \ref{lem:aniso equiv}, one can show that $\kk_i(\Delta)$ are all anisotropic for $0\leq i\leq d$.

We first consider the case that $\dim\Delta=d-1$ is odd.  Set $n=d/2$ and $\sigma_1=[n-1]\cup\{m\}$. Then $\kk'(\Delta')_n$ has a basis of the form: $\{\sigma_1,\sigma_2,\dots,\sigma_s\}$, where $s=h_n(\Delta')$, $\sigma_i\in\Delta'\setminus\mathrm{st}_{\{m\}}\Delta'$ for $2\leq i\leq s$. This follows from the following short exact sequence:
\[0\to\kk'(\Delta',\mathrm{st}_{\{m\}}\Delta')\to\kk'(\Delta')\to\kk'(\mathrm{st}_{\{m\}}\Delta')\to0.\]
Write a nonzero element $\alpha\in\kk'(\Delta')_n$ as $\alpha=\sum_{i=1}^s l_i\xx_{\sigma_i}$, $l_i\in\kk'$. If $l_1=0$, then $\alpha\in\kk'(\Delta',\mathrm{st}_{\{m\}}\Delta')=\kk'(\Delta,\sigma)\subset\kk'(\Delta)$, thus $\alpha^2\neq 0$ since $\kk'(\Delta)$ is anisotropic by the previous paragraph. 
On the other hand, if $l_1\neq 0$, we may assume $l_1=1$. Then, since $\xx_{\sigma_1}\xx_{\sigma_i}=0$ for all $2\leq i\leq s$, we have
\begin{equation}\label{eq:square}
\alpha^2=\xx_{\sigma_1}^2+(l_2\xx_{\sigma_2}+\cdots+l_s\xx_{\sigma_s})^2.
\end{equation}	
An easy calculation shows that in $\kk'(\Delta')_d$, \[\xx_{\sigma_1}^2=-\frac{\prod_{i=1}^{n-1}a_{m,i}}{\prod_{i=n}^{d-1}a_{m,i}}x_1x_2\cdots x_{d-1}x_{m}.\]
Hence $\Psi_{\Delta'}(\xx_{\sigma_1}^2)=\pm \prod_{i=1}^{n-1}a_{m,i}/\prod_{i=n}^{d}a_{m,i}$ by the definition of canonical function. 
Set $\beta=l_2\xx_{\sigma_2}+\cdots+l_s\xx_{\sigma_s}$. If $\alpha^2=0$, then by \eqref{eq:square},
\begin{equation}\label{eq:fraction}
\Psi_{\Delta'}(\beta^2)=\Psi_{\Delta}(\beta^2)=\pm\frac{\prod_{i=1}^{n-1}a_{m,i}}{\prod_{i=n}^{d}a_{m,i}}.
\end{equation}
Let $R=\kk[a_{m,1},\dots,a_{m,d}]$, and let $\mathfrak{p}$ be the prime ideal $(a_{m,d})\subset R$. Then as in the proof of Lemma \ref{lem:aniso equiv}, there is a ring homomorphism $\eta:R_\mathfrak{p}\to \kk_{d-1}$ given by $\eta(a_{m,i})=a_{m,i}$ for $1\leq i\leq d-1$ and  $\eta(a_{m,d})=0$. 
Since $\Psi_{\Delta'}(\xx_{\sigma_i}\xx_{\sigma_j})\in\kk$ for all $2\leq i,j\leq s$, it follows from \eqref{eq:fraction} that there exits a power $b=a_{m,d}^k$ ($k\in\Zz^+$) such that $bl_i\in R_\mathfrak{p}$ for all $2\leq i\leq s$ and $\eta(bl_j)\neq 0$ for some $j$. So \[0\neq\eta(b\beta)\in\kk_{d-1}(\Delta',\mathrm{st}_{\{m\}}\Delta')\subset\kk_{d-1}(\Delta),\] 
and then $\eta(b^2\beta^2)=(\eta(b\beta))^2\neq0$, since $\kk_{d-1}(\Delta)$ is anisotropic. This implies that 
$\Psi_\Delta\eta(b^2\beta^2)=\eta(b^2\Psi_\Delta(\beta^2))\neq 0$. However, the equation \eqref{eq:fraction}, together with the fact that $a_{m,d}^2\mid b^2$, gives $\eta(b^2\Psi_\Delta(\beta^2))=0$, a contradiction. 
Thus we get the anisotropy of $\kk'(\Delta')$ in degree $n$.
If $0\neq\alpha\in\kk'(\Delta')_i$ for $i<n$, then there exits $\alpha'\in\kk'(\Delta')_{n-i}$ such that $0\neq\alpha\alpha'\in\kk'(\Delta')_n$.
This is because $\kk'(\Delta')$ is Gorenstein. So we reduce to the case when $i=n$.

For the case that $\dim\Delta=d-1$ is even, set $n=(d-1)/2$. Then one can similarly show that $\kk'(\Delta')_n$ has a basis of the form: $\{\sigma_1,\sigma_2,\dots,\sigma_s\}$, where $s=h_n(\Delta')$, $\sigma_1=[n-1]\cup\{m\}$ and $\sigma_i\in\Delta'\setminus\mathrm{st}_{\{m\}}\Delta'$ for $2\leq i\leq s$.
As before, it suffices to verify that $\alpha^2\neq0$ for any nonzero element $\alpha=\sum_{i=1}^{s}l_i\xx_{\sigma_i}\in\kk'(\Delta')_n$ with $l_1=1$. Note that $x_m\alpha^2=x_m\xx_{\sigma_1}^2$, and 
\[\Psi_{\Delta'}(x_m\xx_{\sigma_1}^2)=\pm\frac{\prod_{i=1}^{n-1}a_{m,i}}{\prod_{i=n}^{d}a_{m,i}}\neq 0.\]
Then the statement follows immediately.
\end{proof}
Here is another application of Lemma \ref{lem:aniso equiv}, which shows that Theorem \ref{thm:PP} reduces to odd dimensional spheres.

\begin{thm}\label{thm:odd}
	Let $\Delta$ be a $\mathbb{F}$-homology sphere of dimension $2(n-1)$, and denote by $S\Delta:=\partial\Delta^1*\Delta$ the suspension of $\Delta$. 
If $S\Delta$ is generically anisotropic over $\mathbb{F}$, then $\Delta$ is also generically anisotropic over $\mathbb{F}$.
\end{thm}
\begin{proof}
	Suppose that $\FF_0(\Delta)=[m]$, and write $S\Delta=K\cup K'$, where $K=\{v\}*\Delta$ and $K'=\{v'\}*\Delta$. Let $\kk$ be the rational function field 
	\[\mathbb{F}(a_{ij}:1\leq i\leq 2n,\,1\leq j\leq m),\]
	and let $\kk_0\subset\kk$ be the subfield 
	\[\mathbb{F}(a_{ij}:2\leq i\leq 2n,\,1\leq j\leq m).\]
	Choose an l.s.o.p. $\Theta=(\theta_1,\theta_2,\dots,\theta_{2n})$ for $\kk[S\Delta]$ such that $M_\Theta=(A\mid\boldsymbol{\lambda}_v,\boldsymbol{\lambda}_{v'})$, where $A=(a_{ij})$, and $(\boldsymbol{\lambda}_v,\boldsymbol{\lambda}_{v'})=(I_2\mid 0)^T$.
	 Let $\Theta_0=(\theta_2,\theta_3,\dots,\theta_{2n})$. Clearly, $\Theta_0$ restricted to $\Delta$ is an l.s.o.p. for $\kk[\Delta]$.
	It is known that there are two isomorphisms:
	\begin{gather*}
		\kk(K;\Theta)\cong\kk(\Delta;\Theta_0),\ x_i\mapsto x_i\text{ for }1\leq i\leq m,\ x_v\mapsto x_v-\theta_{1};\\
		\kk(\Delta;\Theta_0)_*\cong\kk(K,\Delta;\Theta)_{*+1},\ \alpha\mapsto x_v\alpha
	\end{gather*}
	(see e.g. \cite[Lemma 3.2 and 3.3]{A18}). Hence for a nonzero element $\alpha\in\kk(\Delta;\Theta_0)$, we have $0\neq x_v\alpha\in\kk(K,\Delta;\Theta)$. Note that $K$ is a homology ball with boundary $\Delta$.
	
	Assume  $S\Delta$ is generically anisotropic over $\mathbb{F}$. Then $\kk(S\Delta;\Theta)$ is anisotropic by the proof of Lemma \ref{lem:aniso equiv}. For any nonzero element $\alpha\in\kk_0(\Delta;\Theta_0)_i\subset \kk(\Delta;\Theta_0)_i$, $i\leq n-1$, the second isomorphism above shows that $0\neq x_v\alpha\in\kk(K,\Delta;\Theta)_{i+1}$.  Hence we have  
	$0\neq(x_v\alpha)^2\in\kk(K,\Delta;\Theta)$, since $\kk(K,\Delta;\Theta)=\kk(S\Delta,K';\Theta)\subset\kk(S\Delta;\Theta)$ and $\kk(S\Delta;\Theta)$ is anisotropic. 
	By Proposition \ref{prop:pairing}, this means that $x_v\alpha^2$ is not zero in $\kk(K;\Theta)$, and then $0\neq\alpha^2\in\kk_0(\Delta;\Theta_0)\subset \kk(\Delta;\Theta_0)$ because of the above isomorphism $\kk(K;\Theta)\cong\kk(\Delta;\Theta_0)$. So $\kk_0(\Delta;\Theta_0)$ is anisotropic. This is equivalent to saying that $\Delta$ is generically anisotropic over $\mathbb{F}$.
\end{proof}

\section{A proof of Theorem \ref{thm:PP} for PL-spheres}\label{sec:char 2}
In this section $\mathbb{F}$ denotes a field of characteristic $2$.

By Theorem \ref{thm:odd} and  Theorem \ref{thm:pachner}, if we can show that the characteristic $2$ anisotropy of odd dimensional PL-spheres is preserved by bistellar moves, then Theorem \ref{thm:PP} holds for all PL-spheres. 

\begin{thm}\label{thm:bistellar}
	Let $\Delta$ be a $\mathbb{F}$-homology $(2n-1)$-sphere and suppose that $\Delta'$ is obtained from $\Delta$ via a bistellar $q$-move. Then
	$\Delta$ is generically anisotropic over $\mathbb{F}$ if and only if $\Delta'$ is  generically anisotropic over $\mathbb{F}$.
\end{thm}
\begin{proof}
The case $q=0$ or $2n-1$ is Theorem \ref{thm:0-move}. So we assume $q\neq0,\,2n-1$. Let $\kk$ be the field extension of $\mathbb{F}$, and $\Theta$ be the l.s.o.p. for $\kk[\Delta]$ and $\kk[\Delta']$, as defined at the beginning of section \ref{sec:equivalence}.	Assume $\kk(\Delta;\Theta)$ is anisotropic. We need to show that $\kk(\Delta';\Theta)$ is also anisotropic. As we have seen in the proof of Theorem \ref{thm:0-move}, it suffices to show that  $\kk(\Delta';\Theta)$ is anisotropic in degree $n$.
	
	Suppose that $\FF_0(\Delta)=\FF_0(\Delta')=[m]$, and $\Delta'=\chi_\sigma\Delta$ for some $\sigma\in\Delta$ with $\mathrm{lk}_\sigma\Delta=\partial\Delta^q=\partial\tau$. Then there are short exact sequences:
	\begin{gather}
		0\to \kk(\Delta,\mathrm{st}_\sigma\Delta)\to\kk(\Delta)\to\kk(\mathrm{st}_\sigma\Delta)\to0,\label{eq:delta}\\
			0\to \kk(\Delta',\mathrm{st}_\tau\Delta')\to\kk(\Delta')\to\kk(\mathrm{st}_\tau\Delta')\to0.
			\label{eq:delta'}
	\end{gather}
	Since $\kk(\mathrm{st}_\tau\Delta')_i=0$ for $i\geq\dim\sigma+1=2n-q$,
	it follows from \eqref{eq:delta} and \eqref{eq:delta'} that if $q\geq n$, then  
	\[\kk(\Delta')_{\geq n}=\kk(\Delta',\mathrm{st}_\tau\Delta')_{\geq n}=\kk(\Delta,\mathrm{st}_\sigma\Delta)_{\geq n}\subset\kk(\Delta)_{\geq n}.\] 
	Hence $\kk(\Delta')$ is automatically anisotropic  when $q\geq n$. It remains to consider the case when $0<q<n$.
	
	Without loss of generality we assume that $\tau=[q+1]$, $\sigma=[2n+1]\setminus[q+1]$.
	Let $\kk_0$ be the rational function field
	\[\mathbb{F}(a_{ij}:1\leq i\leq 2n,\,2n+2\leq j\leq m),\]
	and let $\kk_1$ be the rational function field $\kk_0(b_1,\dots,b_{2n})$.
	Suppose that $\Theta_1=(\theta_1,\dots,\theta_{2n})$ is a sequence  of linear forms of $\kk_1[x_1,\dots,x_m]$ such that the column vectors of $M_\Theta$ are given by
	\[\boldsymbol{\lambda}_{j}=
	\begin{cases}
	(a_{1,j},a_{2,j}\dots,a_{2n,j})^T  &\text{ for } 2n+2\leq j\leq m,\\
	(b_1,\dots,b_{2n-q-1},0,\dots,0)^T  &\text{ for } j=2n+1,\\
	(\lambda_{1,j},\lambda_{2,j},\dots,\lambda_{2n,j})^T  &\text{ for } 1\leq j\leq 2n,
	\end{cases}\]
	where 
	\[\lambda_{ij}=
	\begin{cases}
	 1  &\text{ if } i=j,\\
	b_i  &\text{ if } i-j=2n-q-1,\\
0  &\text{ otherwise.} 
	\end{cases}\]
It is easy to verify the following facts: 
\begin{enumerate}[(a)]
	\item $\Theta_1$ is an l.s.o.p. for both $\kk_1[\Delta]$ and $\kk_1[\Delta']$;
	\item\label{it:b} $\kk_1(\mathrm{st}_\tau\Delta';\Theta_1)_n=\kk_1$ is spanned by the face monomial $\xx_{\sigma_1}:=x_1\cdots x_n$;
	\item\label{it:c} \[\xx_{\sigma_1}^2=\frac{\prod_{i=q+2}^nb_i}{\prod_{i=n+1}^{2n}b_i}\cdot\prod_{i=1}^{2n}x_i\in\kk_1(\Delta';\Theta_1)_{2n}.\]
\end{enumerate}  
Here we set $\prod_{i=q+2}^nb_i=1$ if $q=n-1$. 

Similar to Lemma \ref{lem:aniso equiv}, one can show that $\kk(\Delta;\Theta)$ (resp. $\kk(\Delta';\Theta)$) is anisotropic if and only if $\kk_1(\Delta;\Theta_1)$ (resp. $\kk_1(\Delta';\Theta_1)$) is anisotropic. The rest of the proof is to  derive the anisotropy of $\kk_1(\Delta';\Theta_1)$ from the anisotropy of $\kk_1(\Delta;\Theta_1)$. 

By \eqref{eq:delta'} and fact \eqref{it:b}, $\kk_1(\Delta')_n$ has a basis $\{\sigma_1,\sigma_2,\dots,\sigma_s\}$ ($s=h_n(\Delta')$) with $\sigma_i\in \Delta'\setminus\mathrm{st}_\tau\Delta'$ for $2\leq i\leq s$. 
For a nonzero element $\alpha\in\kk_1(\Delta')_n$, write \[\alpha=\sum_{i=1}^sl_i\xx_{\sigma_i},\ l_i\in\kk_1.\]
 If $l_1=0$, then $\alpha\in\kk_1(\Delta',\mathrm{st}_\tau\Delta')=\kk_1(\Delta,\mathrm{st}_\sigma\Delta)\subset \kk_1(\Delta)$. Since $\kk_1(\Delta)$ is anisotropic, we have $0\neq\alpha^2\in\kk_1(\Delta',\mathrm{st}_\tau\Delta')\subset \kk_1(\Delta')$ in this case. 
 
 To deal with the case $l_1\neq0$, define a partial differential operator 
 \[\PP:=\frac{\partial^n}{\partial b_{2n-q}\cdots\partial b_{2n}}.\]
 Then, we have
 \[\PP\Psi_{\Delta'}(\alpha^2)=\PP\sum_{i=1}^s l_i^2\Psi_{\Delta'}(\xx_{\sigma_i}^2)=\sum_{i=1}^s l_i^2\PP\Psi_{\Delta'}(\xx_{\sigma_i}^2),\]
 where the first and second equality both come from the assumption that $\mathrm{char}\,\mathbb{F}=2$.
Since $\tau=[q+1]\not\in \Delta'\setminus\mathrm{st}_\tau\Delta'$, it follows that for any $2\leq i\leq s$, there exist $1\leq n_i\leq q+1$ and a facet $F_i$ in $\mathrm{st}_{\sigma_i}\Delta'$ such that $n_i\not\in F_i$. This implies that the variable $b_{m_i}$, where $m_i=n_i+2n-q-1$, dose not appear in the rational functions $A_{F_i}$ and $A_{F_i}(j)$ for  any $2\leq i\leq s$ and $j\in F_i$ (see the discussion preceding Theorem \ref{thm:Lee} for the definitions of $A_{F_i}$ and $A_{F_i}(j)$). 
Note that  $2n-q\leq m_i\leq 2n$.
Therefore $\PP\Psi_{\Delta'}(\xx_{\sigma_i}^2)=0$ for all $2\leq i\leq s$ by Theorem \ref{thm:Lee}, and then $\PP\Psi_{\Delta'}(\alpha^2)=l_1^2\PP\Psi_{\Delta'}(\xx_{\sigma_1}^2)$.
Combining this with fact \eqref{it:c} and the fact that \[\Psi_{\Delta'}(x_1\cdots x_{2n})=\frac{1}{\det(M_{\Theta_1}([2n]))}=1,\]
 we immediately have $\PP\Psi_{\Delta'}(\alpha^2)\neq 0$, and so $\alpha^2\neq 0$.
 
Thus for any nonzero element $\alpha\in\kk_1(\Delta')_n$, we have $\alpha^2\neq0$, finishing the proof of the theorem.
\end{proof}

\section{Anisotropy and Lefschetz property of simplicial 2-spheres}\label{sec:2-sphere}	
In this section $\mathbb{F}$ denotes a field of arbitrary characteristic.

In section \ref{sec:equivalence}, we have seen that the generic anisotropy of odd dimensional spheres implies the generic anisotropy of even dimensional spheres of one lower dimension.
A natural question is that whether the same condition implies the generic anisotropy of even dimensional spheres of one higher dimension. In this section we will show that the answer to this question is yes for PL-spheres.
\begin{thm}\label{thm:general char}
	If every PL-sphere of dimension $2n-1$ is generically anisotropic over $\mathbb{F}$, then every PL-sphere of dimension $2n$ is also generically anisotropic over $\mathbb{F}$, and has the Lefschetz property over any infinite field of the same characteristic as $\mathbb{F}$. 
	\end{thm}
Since every simplical $1$-sphere is generically anisotropic over $\mathbb{F}$ by Theorem \ref{thm:0-move} and all simplicial $2$-sphere are PL, the following corollary is an immediate consequence of Theorem \ref{thm:general char}.
\begin{cor}\label{cor:2-sphere}
	Every simplicial $2$-sphere is generically anisotropic over any field $\mathbb{F}$, and has the Lefschetz property over any infinite field $\kk$.
\end{cor}
Note that the Lefschetz property of simplicial $2$-spheres is also implied by the result in \cite{Murai10}.

Before giving the proof of Theorem \ref{thm:general char}, we state an easy result about rational functions without proof.
Let $\kk$ be a field, and let $\kk(x)$ be the  field of rational functions over $\kk$. For a nonzero element $\phi=f/g\in\kk(x)$ with $f,g\in\kk[x]$,
define the degree of $\phi$ by $\deg(\phi)=\deg(f)-\deg(g)$, and define the leading coefficient of $\phi$ as $L(\phi):=L(f)/L(g)$, where $L(f)$ and $L(g)$ are the leading coefficient of $f$ and $g$ respectively. Moreover, we assume $\deg(0)=-\infty$ and $L(0)=0$ in $\kk(x)$.
\begin{lem}\label{lem:leading}
Let $\kk(x)$ be as above. Then for a nonzero element $\alpha=\sum_{i\in I}\phi_i$ with $\phi_i\in\kk(x)$, we have 
\[\deg(\alpha)\leq M:=\max\{\deg(\phi_i):i\in I\},\]
where equality holds if and only if $\sum_{\deg(\phi_i)=M}L(\phi_i)\neq0$.
\end{lem}

\begin{proof}[Proof of Theorem \ref{thm:general char}]
	Suppose that $\Delta$ is a $2n$-dimensional PL-sphere with vertex set $[m]$, and set $K=\Delta^0*\Delta=\{v\}*\Delta$. Let $\kk$ be the rational function field
	\[\mathbb{F}(a_{ij}:1\leq i\leq 2n+2,\,1\leq j\leq m).\]
	Choose an l.s.o.p. $\Theta$ for $\kk[K]$ such that $M_\Theta=(A\mid\boldsymbol{\lambda}_v)$, where $A=(a_{ij})$ is the $(2n+2)\times m$ matrix, and $\boldsymbol{\lambda}_v=(1,0\dots,0)^T$.
	If we can show that $\kk(K,\Delta;\Theta)$ is anisotropic, then $\Delta$ is generically anisotropic over $\mathbb{F}$ by the argument in the proof of Theorem \ref{thm:odd}, and the second statement of the theorem will follow from the result in \cite[Section 9]{PP20}.
	As we have seen before, it suffices to verify that this property of $\kk(K,\Delta;\Theta)$ is preserved by bistellar moves on $\Delta$. A result similar to Theorem \ref{thm:0-move} reduces us to considering bistellar $q$-moves for $q\neq 0,\,2n$.
	 
	 Suppose that $\Delta'=\chi_\sigma\Delta$ for some $\sigma\in\Delta$ ($\dim\sigma\neq 0,\,2n$) with $\mathrm{lk}_\sigma\Delta=\partial\tau$, and without loss of generality assume that $\sigma\cup\tau=[2n+2]$ and $1\in\sigma$. 
	 Let $K'=\{v\}*\Delta'$. 
	 Define two subfield of $\kk$ as 
	 \[\kk_0=\mathbb{F}(a_{ij}:1\leq i\leq 2n+2,\,2\leq j\leq m),\ \ \kk_1=\kk_0(a_{1,1}),\]
	 and let $\Theta_1=(\theta_1,\theta_2,\dots,\theta_{2n+2})$ be the l.s.o.p. for both $\kk_1[K]$ and $\kk_1[K']$, such that $M_{\Theta_1}$ is obtained from $M_\Theta$ by replacing the column vector $\boldsymbol{\lambda}_1$ with $(a_{1,1},1,0,\dots,0)^T$. 
	 Then by the proof of Lemma \ref{lem:aniso equiv}, $\kk(K,\Delta;\Theta)$ (resp. $\kk(K',\Delta';\Theta)$) is anisotropic if and only if $\kk_1(K,\Delta;\Theta_1)$ (resp. $\kk_1(K',\Delta';\Theta_1)$) is anisotropic. For this reason, we only consider the anisotropy of $\kk_1(K,\Delta;\Theta_1)$ and $\kk_1(K',\Delta';\Theta_1)$ in what follows, and for simplicity, we omit $\Theta_1$ from the notation.
	
	Assume that $\kk_1(K',\Delta')$ is anisotropic. We need to show that $\kk_1(K,\Delta)$ is also anisotropic, i.e., for any nonzero element $\alpha\in\kk_1(K,\Delta)_{n+1}$, $\alpha^2\neq0$. We first construct a basis for $\kk_1(K,\Delta)_{n+1}$.
 Suppose that $\Aa_0=\{\rho_1,\dots,\rho_r\}\subset\mathrm{lk}_{\{1\}}\Delta$, where $r=h_n(\mathrm{lk}_{\{1\}}\Delta)$, forms a basis for $\kk(\mathrm{st}_{\{1\}}\Delta)_n$. Then the short exact sequence 
	\[0\to\kk_1(\Delta,\mathrm{st}_{\{1\}}\Delta)\to\kk_1(\Delta)\to\kk_1(\mathrm{st}_{\{1\}}\Delta)\to0\]
	implies that $\Aa_0$ can be extended to a basis $\Aa_0\cup\BB_0$ for $\kk_1(\Delta)_n$, where $\BB_0=\{\pi_1,\dots,\pi_s\}$ with $\pi_i\in \Delta\setminus\mathrm{st}_{\{1\}}\Delta$, $s=h_n(\Delta)-r$. Let $\sigma_i=\{v\}\cup\rho_i$, $\tau_i=\{v\}\cup\pi_i$, and $\Aa=\{\sigma_1,\dots,\sigma_r\}$, $\BB=\{\tau_1\dots,\tau_s\}$. Then the isomorphism $\kk(\Delta)\xr{\cdot x_v}\kk(K,\Delta)$ shows that $\Aa\cup\BB$ forms a basis for $\kk(K,\Delta)_{n+1}$.

For a nonzero element $\alpha\in\kk_1(K,\Delta)_{n+1}$, write $\alpha=\alpha_1+\alpha_2$, where
\[\alpha_1=\sum_{i=1}^rl_i\xx_{\sigma_i},\ l_i\in\kk_1,\ \text{ and }\ \alpha_2=\sum_{j=1}^sk_j\xx_{\tau_j},\ k_j\in\kk_1.\]
Now consider the value $\Psi_K(\alpha^2)$.
Since $x_v=-\sum_{i=1}^m a_{1,i}x_i$ in $\kk_1(K,\Delta)$, we have
\[\xx_{\sigma_i}\xx_{\sigma_j}=-\sum_{k=1}^m a_{1,k}\cdot x_vx_k\xx_{\rho_i}\xx_{\rho_j} \text{ for } 1\leq i,j\leq r.\]
Thus, setting $L_1=\mathrm{lk}_{\{1\}}\Delta$ and $\Theta_0=(\theta_3,\dots,\theta_{2n+2})$, and using Theorem \ref{thm:Lee} to compute the canonical functions $\Psi_K$ on $\kk_1(K,\Delta)_{2n+2}$ and $\Psi_{L_1}$ on $\kk_0(L_1;\Theta_0)_{2n}$ respectively, we see that 
\begin{equation}\label{eq:degree1}
\Psi_K(\xx_{\sigma_i}\xx_{\sigma_j})=\pm\Psi_{L_1}(\xx_{\rho_i}\xx_{\rho_j})a_{1,1}+b,\ \text{ for some}\ b\in\kk_0.
\end{equation}
Since $\tau_i\not\in\mathrm{st}_{\{1\}} K$ for all $\tau_i\in\BB$, a similar computation shows that 
\begin{equation}\label{eq:degree2} \Psi_K(\xx_{\tau_i}\xx_{\tau_j}),\,\Psi_K(\xx_{\sigma_i}\xx_{\tau_j})\in\kk_0\ \text{ for all } i,j.
\end{equation}
Thinking of $l_i,\,k_j$ as rational functions with coefficients in $\kk_0$, define \begin{gather*}
	m_1:=\max\{\deg(l_i):1\leq i\leq r\},\ \ m_2:=\max\{\deg(k_j):1\leq j\leq s\},\\ 
\beta_1=\sum_{\deg(l_i)=m_1}L(l_i)\xx_{\sigma_i},\  \beta_2=\sum_{\deg(k_j)=m_2}L(k_j)\xx_{\tau_j}.
\end{gather*}
Here $\deg(\cdot)$, $L(\cdot)$ are defined in the discussion preceding Lemma \ref{lem:leading}.

If $m_1<m_2$, then $\beta_2\neq 0$.
The assumption that  $\Delta'=\chi_\sigma\Delta$ and $1\in\sigma$ implies that $K\setminus(\mathrm{st}_{\{1\}} K\cup\Delta)\subset K'\setminus(\mathrm{st}_{\{1\}} K'\cup\Delta')$. Hence 
\[\beta_2\in\kk_1(K,\mathrm{st}_{\{1\}} K\cup\Delta)\subset \kk_1(K',\mathrm{st}_{\{1\}} K'\cup\Delta')\subset \kk_1(K',\Delta').\]
 It follows that $\Psi_K(\beta_2^2)=\Psi_{K'}(\beta_2^2)\neq 0$, since $\kk_1(K',\Delta')$ is anisotropic. By Lemma \ref{lem:leading}, this means that $\deg(\Psi_K(\alpha_2^2))=2m_2$. Furthermore, \eqref{eq:degree1} and \eqref{eq:degree2} give
 \[
 \deg(\Psi_K(\alpha_1^2))\leq 2m_1+1<2m_2,\ \ \deg(\Psi_K(\alpha_1\alpha_2))\leq m_1+m_2<2m_2.
 \]  
Thus $\deg(\Psi_K(\alpha^2))=2m_2$, and so $\alpha^2\neq0$ in this case. 

On the other hand, if $m_1\geq m_2$, then $\beta_1\neq 0$. It follows that \[0\neq\gamma_1:=\sum_{\deg(l_i)=m_1}L(l_i)\xx_{\rho_i}\in \kk_0(L_1;\Theta_0)_{n}.\] 
Hence, if the condition in the theorem holds, then $\Psi_{L_1}(\gamma_1^2)\neq0$, since $L_1$ is a $(2n-1)$-dimensional PL-sphere.
By \eqref{eq:degree1}, \eqref{eq:degree2} and Lemma \ref{lem:leading}, it follows  that \[\deg(\Psi_K(\alpha^2))=\deg(\Psi_K(\alpha_1^2))=2m_1+1.\]
So $\alpha^2\neq0$ still holds in this case, and the proof is complete.
\end{proof}
	
	\bibliography{M-A}
	\bibliographystyle{amsplain}
\end{document}